\documentclass[a4paper,12pt,reqno]{amsart}
\usepackage{amsmath, amssymb,amsthm}
\usepackage{enumerate}
\usepackage{cite}
\numberwithin{equation}{section}
\newtheorem{theorem}{Theorem}[section]

\newtheorem{proposition}{Proposition}[section]
\newtheorem{corollary}{Corollary}[section]
\newtheorem{remark}{Remark}[section]
\newtheorem{definition}{Definition}[section]
\theoremstyle{definition}

\begin{document}
\bibliographystyle{amsplain}
\title{{{
Orthogonal polynomials on the real line
corresponding to a perturbed chain sequence
}}}
\author{
Kiran Kumar Behera
}
\address{
Department of Mathematics,
Indian Institute of Technology, Roorkee-247667,
Uttarakhand, India
}
\email{krn.behera@gmail.com}
\author{
A. Swaminathan
}
\address{
Department of  Mathematics  \\
Indian Institute of Technology, Roorkee-247 667,
Uttarkhand,  India
}
\email{swamifma@iitr.ac.in, mathswami@gmail.com}
\bigskip
\begin{abstract}
In recent years, chain sequences and their perturbations have played a significant role in characterising the orthogonal polynomials both on the real line as well as on the unit circle.
In this note, a particular disturbance of the chain sequence related to orthogonal polynomials having their true interval of orthogonality as a subset of $[0,\infty)$ is studied leading to an important consequence related to the kernel polynomials.
Such perturbations are shown to be related to transformations of symmetric measures.
An illustration using the generalized Laguerre polynomials is also provided.
\end{abstract}
\maketitle
\pagestyle{myheadings}
\markboth
{Kiran Kumar Behera and A. Swaminathan}
{Orthogonal polynomials on the real line
 corresponding to a perturbed chain sequences}

\section{Preliminaries}
Let $\mathcal{L}$ be a moment functional
defined on the vector space of polynomials
with real coefficients. The quantities
$\mu_k=\mathcal{L}[x^k]$, $k\geq0$
are called the moments of order $k$ and are used to
construct the Hankel matrix
$H=(\mu_{i+j})_{i,j=0}^{\infty}$.
In case the principal submatrices $H_n$ are non-singular,
$\mathcal{L}$ is said to be quasi-definite and
there always exists a
sequence of polynomials
$\{p_n(x)\}_{n=0}^{\infty}$,
where for $n\geq0$,
\begin{align*}
p_{n}(x)=c_{n,n}x^n+c_{n,n-1}x^{n-1}+
           \cdots+c_{n,1}x+c_{n,0},
\quad c_{n,n}>0,
\end{align*}
satisfying the property,
\begin{align*}
\mathcal{L}[p_n(t)p_m(t)]
=\delta_{m,n},
\quad m,n\geq0.
\end{align*}
The sequence
$\{p_n(x)\}_{n=0}^{\infty}$
is said to be orthonormal
with respect to $\mathcal{L}$.
%and has infinite number of terms
%if $\mu(t)$ has infinite points
%of increase in $\mathbb{E}$.
%In case $\mu(t)$ has only a finite $N$ points
%of increase in $\mathbb{E}$,
%$d\mu:=d\mu_{N}$ is called a discrete measure
%and the sequence consists of only $N$ polynomials
%$p_0(x), p_1(x),
%\cdots, p_{N-1}(x)$.

Following the notation used in
\cite[Page 12]{Simon-book-1},
the sequence
$\{p_n(x)\}_{n=0}^{\infty}$
satisfies the three
term recurrence relation
(the so called Favard's Theorem),
\begin{align}
a_{n+1}p_{n+1}(x)=
(x-b_{n+1})p_{n}(x)-a_{n}p_{n-1}(x),
\quad n\geq0,
\label{eqn:three term recurrence relation for orthonormal}
\end{align}
%
%with the initial conditions
%$p_{-1}(x)=0$ and $p_{0}(x)=1$.
where $p_{-1}(x)$ and $p_{0}(x)$
are pre-defined.
Further, $b_{n}\in\mathbb{R}$
and $a_n>0$, $n\geq1$
and are related to the coefficients
of $p_n(x)$ as
\cite{garza-paco-szegotrans-asso-poly-2009},
\begin{align*}
a_n=\frac{c_{n-1,n-1}}{c_{n,n}}
\quad\mbox{and}\quad
b_{n+1}=\frac{c_{n,n-1}}{c_{n,n}}-
\frac{c_{n+1,n}}{c_{n+1,n+1}},
\quad
n\geq1.
\end{align*}
Moreover, though $a_0$ is arbitrary,
it is usually taken
to be equal to $\mathcal{L}[1]=\mu_0>0$ for convenience.
It may be noted that the recurrence coefficients
$\{b_{n+1}\}$ and $\{a_n\}$ satisfy the above conditions
if and only if $\mathcal{L}$ is positive definite, that is,
if $\mathcal{L}[r(x)]>0$ for every polynomial $r(x)$
such that $r(x)>0$ for all $x\in\mathbb{R}$.

In many cases, the polynomials $p_n(x)$
are normalized
to make their leading coefficient equal to one.
These monic orthogonal polynomials denoted as
$\{\mathcal{P}_n(x)\}_{n=0}^{\infty}$
and defined by
$c_{n,n}\mathcal{P}_n(x)=p_{n}(x)$,
$n\geq0$,
satisfy the recurrence relation
\begin{align}
\mathcal{P}_{n+1}(x)=
(x-b_{n+1})\mathcal{P}_{n}(x)-
a_{n}^{2}\mathcal{P}_{n-1}(x),
\quad n\geq0,
\label{eqn:three term recurrence relation for orthogonal}
\end{align}
with $\mathcal{P}_{-1}(x)=0$ and
$\mathcal{P}_{0}(x)=1$.
In the sequel, these will be called an orthogonal
polynomials sequence (OPS).
The system
\eqref{eqn:three term recurrence relation for orthogonal}
is usually written in the more precise form:
$x\mathbb{\bar{P}}=J\mathbb{\bar{P}}$, where
$\mathbb{\bar{P}}=
[\mathcal{P}_0(x)\,\, \mathcal{P}_1(x)\,\, \mathcal{P}_2(x)\,\, \cdots]^{T}$
and
\begin{align}
\label{eqn:monic Jacobi matrix for monic OP}
J=
\left(
  \begin{array}{ccccc}
    b_1    & 1      & 0      & 0      & \cdots \\
    a_1^2  & b_2    & 1      & 0      & \cdots \\
    0      & a_2^2  & b_3    & 1      & \cdots \\
    0      & 0      & a_3^2  & b_4    & \cdots \\
    \vdots & \vdots & \vdots & \vdots & \ddots \\
  \end{array}
\right)
\end{align}
is called the monic Jacobi matrix associated with
the orthogonal polynomials generated by
\eqref{eqn:three term recurrence relation for orthogonal}.
The polynomials defined as
%
%\begin{align*}
%\mathcal{Q}_n(y)=
%\int_{\mathbb{E}}\dfrac{\mathcal{P}_n(y)-\mathcal{P}_n(t)}{y-t}d\mu(t),
$\mathcal{L}[(\mathcal{P}_n(y)-\mathcal{P}_n(t))/(y-t)]$,
$n\geq0$,
%\end{align*}
%
are called the associated polynomials
associated with $\mu$ and
constitute a second solution of the
recurrence relation
\begin{align*}
z_{n+1}=(x-b_{n+1})z_n-a_n^{2}z_{n-1},
\quad n\geq1,
\end{align*}
with the modified initial conditions
%$\mathcal{Q}_{0}(y)=0$
%and $\mathcal{Q}_{1}(y)=1$
$z_0=0$ and $z_1=1$.
Following the notation used in \cite{Chihara_book},
these polynomials will be denoted as $\mathcal{P}^{(1)}_n(x)$, $n\geq0$.
By induction, it can be seen that
$\mathcal{P}^{(1)}_n(x)$, $n\geq1$
is a monic polynomial of degree $n-1$.

It is interesting to note that if $\mathcal{L}$
is positive definite, it can be used to define an
inner product on the space of polynomials. In such
a case, there exists a positive measure $\mu(t)$
supported on a subset $\mathbb{E}$
of the real line $\mathbb{R}$ with $\mathbb{\bar{E}}$ denoting the smallest closed interval containing this support. The linear functional now has the
representation $\mathcal{L}[t^k]=\int_{\mathbb{E}}t^kd\mu(t)$.
Further, the Stieltjes
transform of the measure
\begin{align*}
\mathcal{S}(x)=\int_{\mathbb{E}}\dfrac{d\mu(t)}{x-t},
\end{align*}
yields the Jacobi continued fraction expansion
%
%\begin{align}
%\int_{\mathbb{E}}\dfrac{d\mu(t)}{x-t}\thicksim
%\dfrac{1}{x-\sigma_0-\dfrac{\tau_1^2}{x-\sigma_1-
%\dfrac{\tau_2^2}{x-\sigma_2-\ddots}}}
%
%\end{align}
%%%%%%%%%
%
\begin{align}
\int_{\mathbb{E}}\dfrac{d\mu(t)}{x-t}=\dfrac{1}{x-b_1}
\begin{array}{cc}\\$-$\end{array}
\dfrac{a_1^2}{x-b_2}
\begin{array}{cc}\\$-$\end{array}
\dfrac{a_2^2}{x-b_3}
\begin{array}{cc}\\$-$\end{array}
\ldots
\label{eqn:Jacobi continued fraction definition}
\end{align}
where the $\{a_n^2\}$ and $\{b_n\}$
are the same coefficients appearing in
\eqref{eqn:three term recurrence relation for orthogonal}.
The Stieltjes function $\mathcal{S}(x)$
plays a very fundamental role in the theory of
orthogonal polynomials on the real line.
It admits the power series
expansion at infinity
\begin{align*}
\mathcal{S}(x)=
\dfrac{\mu_0}{x}+\dfrac{\mu_1}{x^2}+\dfrac{\mu_2}{x^3}+\cdots
\end{align*}
and hence acts as the generating
function for the moments $\mu_k's$
associated with $d\mu(t)$.
Further,
the $n^{th}$ convergent of the continued fraction
in the right hand side of
\eqref{eqn:Jacobi continued fraction definition}
is easily seen to be
$\mathcal{P}^{(1)}_n(x)/\mathcal{P}_n(x)$
and hence as
$n\rightarrow\infty$,
$\mathcal{P}^{(1)}_n(x)/\mathcal{P}_{n}(x)$
converges uniformly to $\frac{1}{\mu_0}\mathcal{S}(x)$
on compact subsets of $\mathbb{\bar{C}}\setminus\mathbb{\bar{E}}$
\cite{Ranga-chain-seq-symmetric-OP-JCAM-2002}. Here
$\mathbb{\bar{C}}$ is the extended complex plane.

An equivalence transformation of
the $n^{th}$ convergent of the continued fraction in
\eqref{eqn:Jacobi continued fraction definition}
yields \cite{Ranga-chain-seq-symmetric-OP-JCAM-2002}
%
%\begin{align*}
%\mathcal{F}_{n}(x)=
%\dfrac{1}{1-\dfrac{d_1(x)}{1-\dfrac{d_2(x)}
%{\dfrac{\ddots}{1-\dfrac{d_{n-1}(x)}{1}}}}}
%\end{align*}
%%%%%%%%
%
\begin{align*}
\mathcal{F}_n(x)=
\dfrac{1}{1}
\begin{array}{cc}\\$-$\end{array}
\dfrac{d_1(x)}{1}
\begin{array}{cc}\\$-$\end{array}
\dfrac{d_2(x)}{1}
\begin{array}{cc}\\$-$\end{array}
\ldots
\begin{array}{cc}\\$-$\end{array}
\dfrac{d_{n-1}(x)}{1}
\end{align*}
where
\begin{align*}
\mathcal{F}_n(x)=
(x-b_1)\dfrac{\mathcal{P}^{(1)}_n(x)}{\mathcal{P}_{n}(x)}
\quad\mbox{and}\quad
d_{n}(x)=\dfrac{a_n^2}{(x-b_{n})(x-b_{n+1})}.
\end{align*}
The quantities $d_n(t)$ can also be obtained,
\cite[Page 110]{Chihara_book},
from the recurrence relation
\eqref{eqn:three term recurrence relation for orthogonal}
as
\begin{align}
\label{eqn:min parameters from ttrr chihara theorem}
d_n(t)=\dfrac{\mathcal{P}_n(t)}
             {(t-b_{n})\mathcal{P}_{n-1}(t)}
       \left[1-\dfrac{\mathcal{P}_{n+1}(t)}
                     {(t-b_{n+1})\mathcal{P}_{n}(t)}
       \right],
\quad n\geq1.
\end{align}
Such structures,
called chain sequences have been studied in
\cite{Wall_book},
followed by a systematic treatment in
\cite{Chihara_book}
(see also \cite[Section 7.2]{Ismail_book}),
particularly in the context of orthogonal
polynomials on the real line.

Formally stated, a sequence
$\{d_n\}_{n=1}^{\infty}$
is called a positive chain sequence if
it can be expressed in terms of another sequence
$\{g_n\}_{n=0}^{\infty}$
as
$d_n=(1-g_{n-1})g_n$, $n\geq1$.
Here,
$\{g_n\}_{n=0}^{\infty}$
called a parameter sequence of the chain sequence
$\{d_n\}_{n=1}^{\infty}$
is such that $0\leq g_0<1$ and
$0<g_n<1$, $n\geq1$.
The parameter sequence is in general not unique,
and in such cases
it is interesting to study the
bounds for each parameter
$g_n$, $n\geq0$.
The minimal parameter sequence
$\{m_n\}_{n=0}^{\infty}$
of the positive chain sequence $\{d_n\}$
is defined as $m_0=0$, $0<m_n<1$,
with $d_n=(1-m_{n-1})m_n$, $n\geq1$.
Thus, for instance,
the minimal parameter sequence of the chain sequence
$\{d_n(t)\}_{n=1}^{\infty}$
as defined in
\eqref{eqn:min parameters from ttrr chihara theorem}
is given by
$m_n(t)=
[1-\mathcal{P}_{n+1}(t)/(t-b_{n+1})\mathcal{P}_n(t)]$,
$n\geq0$.
At the other end,
there is the maximal parameter sequence
$\{M_n\}_{n=0}^{\infty}$,
whose parameters are defined as,
\begin{align*}
M_k=
\sup\{g_k:\{g_n\}\mbox{ any other parameter sequence of }\{d_n\}\}
\end{align*}
where $\sup{}$ denotes the supremum of the set.
It is known that
$m_k\leq g_k\leq M_k$, $k\geq0$,
for any parameter sequence
$\{g_n\}$ of $\{d_n\}$.
In case, $M_0=m_0=0$,
the parameter sequence is
unique and $\{d_n\}$ is called a
single parameter positive chain sequence,
abbreviated as SPPCS in the sequel.
\section{Symmetric Orthogonal Polynomials}
A positive measure $d\phi(x)$ is called symmetric if it satisfies
$d\phi(-x)=-d\phi(x)$. The monic polynomials $\mathcal{S}_n(x)$, $n\geq0$, which are orthogonal with respect to $d\phi(x)$ satisfy
\begin{align*}
\mathcal{S}_n(x)=x\mathcal{S}_{n-1}(x)-
\gamma_n^{\phi}\mathcal{S}_{n-2}(x),
\quad n\geq1,
\end{align*}
with $\mathcal{S}_{-1}(x)=0$ and $\mathcal{S}_0(x)=1$.

In a series of research articles
\cite{Ranga-Berti-companion-OP-applications,
Ranga-symmetric-OP-laurent-polynomials,
Ranga-companion-orthogonal-polynomial,
Ranga-McCabe-symmetry-strong-distributions},
many properties of symmetric orthogonal polynomials have been discussed in the context of certain orthogonal Laurent polynomials. Particularly, if $d\phi_1(x)$ and $d\phi_2(x)$ are two different symmetric distributions related by the Christoffel transformation as
$d\phi_2(x)=(1+kx^2)d\phi_1(x)$, where $k$ is positive real constant, then
\begin{theorem}\textup{\cite[Theorem 1]{Ranga-companion-orthogonal-polynomial}}
  Associated with $d\phi_1(x)$ and $d\phi_2(x)$ there exists a sequence of positive numbers $\{l_n\}$, with $l_0=1$ and $l_n>1$ for $n\geq1$, such that
  \begin{align}
  \label{eqn: alpha_n in terms of ln}
  4k\gamma_{n+1}^{\phi_1}=(l_n-1)(l_{n-1}+1)
  \quad\mbox{and}\quad
  4k\gamma_{n+1}^{\phi_2}=(l_n-1)(l_{n+1}+1)
  \end{align}
\end{theorem}
This result was established using the transformation
$t(x)=(\sqrt{\alpha x^2+\beta}+\sqrt{\alpha}x)^2$ for $x\in(-\infty, \infty)$ and the monic polynomials $\{\mathcal{B}_n(t)\}_{n=0}^{\infty}$,
uniquely defined by
\begin{align*}
\int_{0}^{\infty}t^{-n+s}\mathcal{B}_n(t)d\psi(t)=0,
\quad
0\leq s \leq n-1,
\end{align*}
where $d\psi(t)$ is strong distribution on $(0,\infty)$. Note that if $x=d$
corresponds to $t=b$, the transformation maps $(-d,d)$ to
$(\beta^2/b,b)\subseteq (0,\infty)$. For any strong distribution $d\psi(t)$ on $(0,\infty)$, the monic polynomials $\{\mathcal{B}_{n}(t)\}_{n=0}^{\infty}$ satisfy
\begin{align*}
\mathcal{B}_{n+1}(t)=(t-\beta_{n+1}^{\psi})\mathcal{B}_{n}(t)-
        \alpha_{n+1}^{\psi}t \mathcal{B}_{n-1}(t),
\quad n\geq1,
\end{align*}
with $\mathcal{B}_{0}(t)=1$ and $\mathcal{B}_{1}(t)=t-\beta_1^{\psi}$. Further, the unique coefficients $\beta_{n+1}^{\psi}$, $n\geq0$ and
$\alpha_{n+1}^{\psi}$, $n\geq1$ are all positive real numbers.

Two kinds of strong distributions play important role in this analysis. The first is the ScS$(a,b)$ distribution and the second is the S$\bar{c}$S$(a,b)$ distribution. A strong distribution $d\psi(t)$ with its support inside $(\beta^2/b,b)$ is called a ScS$(\beta^2/b,b)$ distribution if
\cite{Ranga-companion-orthogonal-polynomial}
\begin{align*}
\dfrac{d\psi(t)}{\sqrt{t}}=\dfrac{d\psi(\beta^2/t)}{\beta/\sqrt{t}}, \quad t\in(\beta^2/b,b).
\end{align*}
A strong distribution $d\psi(t)$ with its support inside
$(\beta^2/b,b)$ is called a S$\bar{c}$S$(\beta^2/b,b)$ distribution if
$d\psi(t)=-d\psi(\beta^2/b)$, $t\in(\beta^2/b,b)$
\cite{Ranga-companion-orthogonal-polynomial}.
Then it is known that
$d\psi_2(t)$ is a ScS$(\beta^2/b,b)$ distribution if and only if
$(t+\beta)d\psi_0(t)$ is a S$\bar{c}$S$(\beta^2/b,b)$ distribution.
Further, if
%$\phi^{\psi}=\beta_n^{\psi}+\alpha_{n+1}^{\psi}$ and
$l_n^2-1=\alpha^{\psi_0}_{n+1}/\beta_n^{\psi_0}$, then
\begin{align*}
\alpha_{n+1}^{\psi_2}=\beta(l_n-1)(l_{n+1}+1),
\quad n\geq1.
\end{align*}
Conversely, $d\psi_0(t)$ is a S$\bar{c}$S$(\beta^2/b,b)$ distribution if and only if $\frac{t+\beta}{t}d\psi_1(t)$ is a ScS$(\beta^2/b,b)$ distribution and in this case
\begin{align*}
\alpha_{n+1}^{\psi_1}=\beta(l_n-1)(l_{n-1}+1),
\quad n\geq1.
\end{align*}
Then choosing
\begin{align*}
d\psi_1(t)=A_1\dfrac{t}{t+\beta}d\phi_1(x(t)),
\quad
d\psi_2(t)=A_2\dfrac{t}{t+\beta}d\phi_2(x(t)),
\end{align*}
with $k=\alpha/\beta$ and appropriate choices of $A_1$ and $A_2$, the relations \eqref{eqn: alpha_n in terms of ln} are obtained.

In this note, we consider the particular case when the sequence $\{l_n\}$ satisfy $l_1\neq l_2$ and
\begin{align*}
l_1&=l_4=l_5=l_8=l_9=\cdots\\
l_2&=l_3=l_6=l_7=l_{10}=\cdots
\end{align*}
This sequence can be obtained from appropriate conditions imposed on the coefficients $\beta^{\psi}_{n+1}$ and $\alpha^{\psi}_{n+1}$, for instance, choosing the coefficients such that they satisfy
$\frac{\alpha_2^{\psi_0}}{\beta_1^{\psi_0}}=
\frac{\alpha_5^{\psi_0}}{\beta_4^{\psi_0}}$ and so on.
Consequently, the coefficients associated with the Christoffel transformation mentioned above satisfy
\begin{align*}
\{\gamma_1^{\phi_2},\gamma^{\phi_2}_2,\gamma^{\phi_2}_3,
\gamma^{\phi_2}_4,\cdots\}
=
\{\gamma^{\phi_1}_2,\gamma^{\phi_1}_1,\gamma^{\phi_1}_4,
\gamma^{\phi_1}_3,\cdots\}
\end{align*}

Thus it can be seen that transformations of the symmetric measures are related to transformations of strong distributions on $(0, \infty)$
which possess special properties like symmetry.
In this note, we study this particular case of the transformation from the point of view of a perturbation of the chain sequences associated with polynomial sequences orthogonal on $(0,\infty)$ and some related consequences. We also provide an illustration using the Laguerre polynomials in the last section.
\section{Perturbed chain sequences}

The theory of chain sequences has been
used to study many properties of
a given orthogonal
polynomial system on the real line,
for instance,
its true interval of orthogonality,
which is the smallest closed interval that contains
all the zeros of all the polynomials of such a system.
Denoting the true interval of orthogonality of
$\{\mathcal{P}_n(x)\}$ satisfying
\eqref{eqn:three term recurrence relation for orthogonal}
as $[\xi, \eta]$
and
\begin{align*}
\omega_n(t)=
\dfrac{a_n^2}{(t-b_{n})(t-b_{n+1})},
\quad
n\geq1,
\end{align*}
the following are known to be equivalent,
\cite[Corollary 7.2.4]{Ismail_book},
\begin{enumerate}[(i)]
\item $[\xi,\eta]$ is contained in $(a,b)$,
\item $b_{n+1}\in(a,b)$ for $n\geq0$
      and both $\{\omega_n(a)\}$ and
      $\{\omega_n(b)\}$ are chain sequences.
\end{enumerate}
We note that here
$t\in (-\infty, \infty)\setminus\mathbb{\bar{E}}$.
In particular, the true interval of
orthogonality is a subset of $[0,\infty]$
if and only if
$b_{n+1}>0$ for $n\geq0$
and there are numbers $g_n$ such that
$0\leq g_0<1$, $0<g_n<1$, $n\geq1$,
satisfying,
\begin{align*}
(1-g_{n-1})g_n=
\dfrac{a_n^2}{b_{n}b_{n+1}}=\omega_{n}(0),
\quad n\geq1.
\end{align*}
As in \cite{Chihara_book},
the sequence
$\{g_n\}_{n=0}^{\infty}$
is constructed using
another sequence
$\{\gamma_n\}_{n=1}^{\infty}$,
where
$\gamma_1\geq0$ and $\gamma_n>0$, $n\geq2$.
For details, the reader is referred to
\cite[Chapter 1, Theorems 9.1, 9.2]{Chihara_book},
where it is shown that
\begin{align*}
b_{n+1}=\gamma_{2n+1}+\gamma_{2n+2},
\quad
n\geq0
\quad\mbox{and}\quad
a_n^2=\gamma_{2n}\gamma_{2n+1},
\quad n\geq1.
\end{align*}
and the parameters are given by $g_n=\gamma_{2n+1}/b_{n+1}$, $n\geq0$.
Hence when $\gamma_1=0$, we obtain the minimal parameter sequence $\{m_n\}_{n=0}^{\infty}$.
%%
%\begin{align*}
%m_{n}=\dfrac{\gamma_{2n+1}}{b_{n+1}}
%\quad\mbox{and}\quad
%1-m_{n}=\dfrac{\gamma_{2n+2}}{b_{n+1}},
%\quad n\geq0.
%\end{align*}
%%

Associated with the chain sequence $\{\omega_n(0)\}$,
another sequence
$\{\tilde{\omega}_n(0)\}$
arises in a very natural way.
%Indeed, $\gamma_{2n+1}>0$ gives
%$b_{n+1}>\gamma_{2n+2}$, $n\geq1$.
%Hence, $\gamma_{2n+2}=k_n b_{n+1}$, for some
%$0<k_n<1$, $n\geq1$.
Defining
$\tilde{\omega}_1(0)=(1-k_0)k_1=\gamma_4/b_2$
and
\begin{align*}
\tilde{\omega}_{n}(0)=
\dfrac{\gamma_{2n-1}\gamma_{2n+2}}{b_{n}b_{n+1}}=
(1-k_{n-1})k_n,
\quad n\geq2,
\end{align*}
it can be seen that
$\{\tilde{\omega}_n\}_{n=1}^{\infty}$
becomes a chain sequence with the minimal
parameter sequence $\{k_n\}_{n=0}^{\infty}$
where $k_0=0$ and $k_n=1-g_n$, $n\geq1$.
Hence, we give the following definition
which was given earlier in
\cite{CCS_OPUC} with reference to orthogonal
polynomials on the unit circle.
\begin{definition}
Suppose $\{d_n\}_{n=1}^{\infty}$
is a chain sequence with
$\{m_n\}_{n=0}^{\infty}$
as its minimal parameter sequence.
Let $\{k_n\}_{n=0}^{\infty}$
be another sequence given by
$k_0=0$ and $k_n=1-m_n$ for $n\geq1$.
Then the chain sequence $\{a_n\}_{n=1}^{\infty}$
having
$\{k_n\}_{n=0}^{\infty}$
as its minimal parameter
sequence is called as
complementary chain sequence of
$\{d_n\}_{n=1}^{\infty}$.
\end{definition}
If $\gamma_1>0$,
a non-minimal parameter sequence
$\{g_n\}_{n=0}^{\infty}$
%given by,
%$g_n=\gamma_{2n+1}/b_{n+1}$ and
%$1-g_n=\gamma_{2n+2}/b_{n+1}$, $n\geq0$
is obtained for the chain sequence
$\{\omega_n(0)\}_{n=1}^{\infty}$.
%say $\{\vartheta_n(0)\}_{n=1}^{\infty}$.
%%
%It is clear that $\vartheta_1(0)\neq \omega_1(0)$
%but
%$\vartheta_n(0)=$, $n\geq2$.
%%
In this case,
%the sequence
%$\{k_n'\}_{n=0}^{\infty}$ where
%$k_n'=1-g_n$, $n\geq0$,
%is constructed,
%so that the associated chain sequence
%is given by
the associated chain sequence
$\{\hat{\vartheta}_n(0)\}_{n=1}^{\infty}$,
is defined as
\begin{align*}
\hat{\vartheta}_n(0)=
(1-k_{n-1}')k_n'=
\dfrac{\gamma_{2n-1}\gamma_{2n+2}}{b_{n}b_{n+1}},
\quad n\geq1.
\end{align*}
where $k_n'=1-g_n$ for $n\geq0$.
\begin{definition}
Suppose $\{d_n\}_{n=1}^{\infty}$
is a chain sequence with
$\{g_n\}_{n=0}^{\infty}$
as its non-minimal parameter sequence.
Let $\{k_n'\}_{n=0}^{\infty}$
be another sequence given by
$k_n'=1-g_n$ for $n\geq0$.
Then the chain sequence
$\{a_n\}_{n=1}^{\infty}$
having
$\{k_n'\}_{n=0}^{\infty}$
as its parameter
sequence is called as
generalised complementary chain sequence
of
$\{d_n\}_{n=1}^{\infty}$.
\end{definition}
It may be noted from the above two definitions that
for a fixed chain sequence, while its
complementary chain sequence is unique,
its generalised complementary chain sequence
need not be unique.
In fact, a chain sequence
will have as many
generalised complementary chain sequences
as its non-minimal parameter sequences.
Naturally, the complementary chain sequence and all the generalised complementary chain sequences will coincide only for a SPPCS.

We would like to mention that the chain sequences
$\{\tilde{\omega}_{n}(0)\}_{n=1}^{\infty}$ and
$\{\hat{\vartheta}_n(0)\}_{n=1}^{\infty}$
have definite sources in the
theory of orthogonal polynomials
on the real line.
To see this, first note that
the symmetric polynomials $\{\mathcal{S}_n(x)\}$ satisfy the property
$\mathcal{S}_n(-x)=(-1)^{n}\mathcal{S}_n(x)$, $n\geq1$, which implies the existence of two OPS $\{\mathcal{P}_n(x)\}_{n=1}^{\infty}$ and $\{\mathcal{K}_n(x)\}_{n=1}^{\infty}$ such that
\begin{align*}
\mathcal{S}_{2n}(x)=\mathcal{P}_n(x^2),
\quad
\mathcal{S}_{2n+1}(x)=x\mathcal{K}_n(x^2).
\end{align*}
It is interesting to note that
$\{\mathcal{K}_n(x)\}=\{\mathcal{K}_n(0;x)\}$
is the sequence of kernel polynomial
corresponding to $\mathcal{P}_n(x)$
based at the origin and abbreviated as KOPS in the sequel.

Further, if the polynomials
$\{\mathcal{R}_n^{(i)}(x)\}$, $i=1,2$,
satisfy the recurrence relation,
\begin{align}
\label{eqn:general ttrr for polynomial and kernel}
\mathcal{R}_{n+1}^{(i)}(x)=
(x-b_{n+1}^{(i)})\mathcal{R}_{n}^{(i)}(x)-
(a_n^2)^{(i)}\mathcal{R}_{n-1}^{(i)}(x),
\quad n\geq0,
\end{align}
with $\mathcal{R}_{-1}^{(i)}(x)=0$
and
$\mathcal{R}_0^{(i)}(x)=1$
then,
\begin{enumerate}[(i)]
\item $\mathcal{R}_n^{(1)}(x)\equiv \mathcal{P}_n(x)$,
$n\geq1$, if and only if
%(setting $\delta_n=\gamma_n$, $n\geq2$),
%
\begin{align*}
b_1^{(1)}&=\gamma_2,
\quad
b_{n+1}^{(1)}=\gamma_{2n+1}+\gamma_{2n+2},
\quad n\geq1\\
%\quad\mbox{and}\quad
(a_n^2)^{(1)}&=\gamma_{2n}\gamma_{2n+1},
%\,\,(\gamma_0:=1)
\quad n\geq1,
\end{align*}
\item $\mathcal{R}_n^{(2)}(x)\equiv \mathcal{K}_n(x)$,
$n\geq1$, if and only if
%(setting $\delta_n=\gamma_{n+1}$, $n\geq1$),
%
\begin{align*}
b_{n+1}^{(2)}=\gamma_{2n+2}+\gamma_{2n+3},
\quad n\geq0
\quad\mbox{and}\quad
(a_n^2)^{(2)}=\gamma_{2n+1}\gamma_{2n+2},
\quad n\geq1.
\end{align*}
\end{enumerate}
With these notations, the parameter sequences can be denoted as
$m_n=\gamma_{2n+1}/b_{n+1}^{(1)}$ and $g_n=\gamma_{2n+1}/b_{n+1}^{(1)}$,
$n\geq0$.
Further, denoting
$\tilde{a}_n^2=\gamma_{2n-1}\gamma_{2n+2}$,
$n\geq1$,
the following theorem shows that the polynomials
$\{\tilde{\mathcal{P}}_n(x)\}$ and
$\{\hat{\mathcal{P}}_n(x)\}$
associated respectively,
with the complementary chain sequence
$\{\tilde{\omega}_n(0)\}$
and the generalised complementary chain sequence
$\{\hat{\vartheta}_n(0)\}$,
can be attributed to a particular perturbation of
the recurrence coefficients of the polynomials
$\{S_n(x)\}$.
\begin{remark}
  Such perturbations of the recurrence coefficients as well
  as of the Stieltjes function have been studied deeply.
  The reader is referred to
  \textup{\cite{castillo-co-polynomials-JMAA-2015,
  garza-paco-szegotrans-asso-poly-2009,
  Zhedanov-rat-spectral-JCAM-1997}}
  for some details.
  In most of the cases only a single modification or a
  finite composition of modifications is considered.
  In this note however, all the recurrence coefficients are
  perturbed.
\end{remark}
\begin{theorem}
\label{thm:main theorem Pn Kn for CCS case}
Let the symmetric polynomials $\{\tilde{S}_n(x)\}_{n=0}^{\infty}$
satisfy
\begin{align}
\tilde{S}_{n}(x)=x\tilde{S}_{n-1}(x)-\tilde{\nu}_n\tilde{S}_{n-2}(x),
\quad n\geq1,
\label{eqn: three term recurrence relation for Sn}
\end{align}
with
$\tilde{S}_{-1}(x)=0$, $\tilde{S}_{0}(x)=1$
and where, for $n\geq1$,
\begin{align}
\label{eqn:perturbation in recurrence coeff for symmetric poly}
\tilde{\nu}_n=
\left\{
  \begin{array}{ll}
    \gamma_{2j-1}, & \hbox{n=2j,\,\, j=1,2,$\cdots$} \\
    \gamma_{2j+2}, & \hbox{n=2j+1,\, j=0,1,$\cdots$.}
  \end{array}
\right.
\end{align}
Then,
with $\gamma_1\neq0$,
$\{\tilde{\mathcal{P}}_n(x)\}_{n=0}^{\infty}$,
where
$\tilde{S}_{2n}(x)=\tilde{\mathcal{P}}_n(x^2)$,
satisfy,
\begin{align}
\label{eqn:ttrr for oprl obtained from CCS}
\tilde{\mathcal{P}}_{n+1}(x)=
(x-b_{n+1}^{(1)})\tilde{\mathcal{P}}_{n}(x)-
\tilde{a}_n^2\tilde{\mathcal{P}}_{n-1}(x),
\quad n\geq1,
\end{align}
with the initial conditions
$\tilde{\mathcal{P}}_{0}(x)=1$ and
$\tilde{\mathcal{P}}_{1}(x)=(x-\gamma_1)$.
\end{theorem}
\begin{proof}
First note that, the perturbation
\eqref{eqn:perturbation in recurrence coeff for symmetric poly}
implies that the sequence of coefficients
$\{\gamma_1,\gamma_2,\gamma_3,\gamma_4,\cdots\}$
is replaced by
$\{\gamma_2,\gamma_1,\gamma_4,\gamma_3,\cdots\}$. That is,
$\{\gamma_{2k-1},\gamma_{2k}\}$ are pair-wise interchanged to
$\{\gamma_{2k},\gamma_{2k-1}\}$, $k\geq1$.
Then, for $n=2m$,
\begin{align*}
\tilde{S}_{2m}(x)=x\tilde{S}_{2m-1}(x)-\gamma_{2m-1}\tilde{S}_{2m-2}(x),
\quad m\geq1
\end{align*}
which implies,
\begin{align*}
\tilde{\mathcal{P}}_{m}(x^2)=x^2\tilde{\mathcal{K}}_{m-1}(x^2)-
\gamma_{2m-1}\tilde{\mathcal{P}}_{m-1}(x^2),
\end{align*}
or equivalently,
\begin{align}
\tilde{\mathcal{P}}_{m}(x)=
x\tilde{\mathcal{K}}_{m-1}(x)-
\gamma_{2m-1}\tilde{\mathcal{P}}_{m-1}(x),
\quad m\geq1.
\label{eqn: recurrence relation for Pn in terms of Qn}
\end{align}
Similarly, for $n=2m+1$,
\begin{align*}
\tilde{S}_{2m+1}(x)=
x\tilde{S}_{2m}(x)-\gamma_{2m+2}\tilde{S}_{2m-1}(x),
\quad m\geq0,
\end{align*}
which implies,
\begin{align*}
x\tilde{\mathcal{K}}_{m}(x^2)=
x\tilde{\mathcal{P}}_{m}(x^2)-
\gamma_{2m+2}x\tilde{\mathcal{K}}_{m-1}(x^2),
\end{align*}
or equivalently,
\begin{align}
\tilde{\mathcal{K}}_{m}(x)=
\tilde{\mathcal{P}}_{m}(x)-
\gamma_{2m+2}\tilde{\mathcal{K}}_{m-1}(x),
\quad m\geq0.
\label{eqn: recurrence relation for Qn in terms of Pn}
\end{align}
Using
\eqref{eqn: recurrence relation for Pn in terms of Qn}
and
\eqref{eqn: recurrence relation for Qn in terms of Pn},
it is easy to find the three term recurrence relations for
$\tilde{\mathcal{P}}_n(x)$ and $\tilde{\mathcal{K}}_n(x)$.
For this, first
$\tilde{\mathcal{P}}_n(x)$
is eliminated. From
\eqref{eqn: recurrence relation for Qn in terms of Pn},
it can be seen that,
\begin{align*}
\tilde{\mathcal{P}}_{m}(x)&=
\tilde{\mathcal{K}}_{m}(x)+\gamma_{2m+2}\tilde{\mathcal{K}}_{m-1}(x),
%\tilde{\mathcal{P}}_{m-1}(x)&=\tilde{\mathcal{K}}_{m-1}(x)+\gamma_{2m}\tilde{\mathcal{K}}_{m-2}(x)
\quad m\geq0.
\end{align*}
Using this in
\eqref{eqn: recurrence relation for Pn in terms of Qn},
gives,
%
%\begin{align*}
%\tilde{\mathcal{K}}_{m}(x)+
%\gamma_{2m+2}\tilde{\mathcal{K}}_{m-1}(x)=
%x\tilde{\mathcal{K}}_{m-1}(x)-
%\gamma_{2m-1}\tilde{\mathcal{K}}_{m-1}(x)-
%\gamma_{2m-1}\gamma_{2m}\tilde{\mathcal{K}}_{m-2}(x)
%\end{align*}
%
%or,
\begin{align}
\tilde{\mathcal{K}}_{m}(x)=
[x-(\gamma_{2m-1}+\gamma_{2m+2})]\tilde{\mathcal{K}}_{m-1}(x)-
\gamma_{2m-1}\gamma_{2m}\tilde{\mathcal{K}}_{m-2}(x),
\quad m\geq1.
\label{eqn: recurrence relation for kernel polynomials in ccs}
\end{align}
with
$\tilde{\mathcal{K}}_{-1}(x)=0$
and (using
\eqref{eqn: recurrence relation for Qn in terms of Pn})
$\tilde{\mathcal{K}}_0(x)=1$.
Similarly,
\eqref{eqn: recurrence relation for Pn in terms of Qn}
gives
\begin{align*}
x\tilde{\mathcal{K}}_{m-1}(x)&=\tilde{\mathcal{P}}_{m}(x)+\gamma_{2m-1}\tilde{\mathcal{P}}_{m-1}(x),\\
%x\tilde{\mathcal{K}}_{m}(x)&=\tilde{\mathcal{P}}_{m+1}(x)+\gamma_{2m+1}\tilde{\mathcal{P}}_{m}(x).
\end{align*}
Using this in
\eqref{eqn: recurrence relation for Qn in terms of Pn}
yields,
%
%\begin{align*}
%\tilde{\mathcal{P}}_{m+1}(x)+\gamma_{2m+1}\tilde{\mathcal{P}}_{m}(x)=x\tilde{\mathcal{P}}_{m}(x)-
%\gamma_{2m+2}\tilde{\mathcal{\mathcal{P}}}_{m}(x)-\gamma_{2m+2}\gamma_{2m-1}\tilde{\mathcal{P}}_{m-1}(x),
%\end{align*}
%%
%or,
%%
\begin{align}
\label{eqn: recurrence relation for polynomials in ccs}
\tilde{\mathcal{P}}_{m+1}(x)=[x-(\gamma_{2m+1}+\gamma_{2m+2})]\tilde{\mathcal{P}}_{m}(x)-
\gamma_{2m-1}\gamma_{2m+2}\tilde{\mathcal{P}}_{m-1}(x),
\quad m\geq1,
\end{align}
%
%or,
%\begin{align}
%\tilde{\mathcal{P}}_{m+1}(x)=[x-\sigma_m^{(1)}]\tilde{\mathcal{P}}_{m}(x)-
%\tilde{\tau}_m^{2}\tilde{\mathcal{P}}_{m-1}(x),
%\quad m\geq1,
%
%\end{align}
%
with the initial conditions
$\tilde{\mathcal{P}}_0(x)=1$ and
(using
\eqref{eqn: recurrence relation for Pn in terms of Qn})
$\tilde{\mathcal{P}}_1(x)=x-\gamma_1$,
thus proving the theorem.
\end{proof}
\begin{corollary}
  Consider the OPS $\{\hat{P}_{n}(x)\}_{n=0}^{\infty}$ satisfying
  \eqref{eqn: recurrence relation for polynomials in ccs}
  but for $m\geq0$. Then $\{\hat{P}_{n}(x)\}_{n=0}^{\infty}$ is associated with the generalised complementary chain sequence $\{\hat{\vartheta}_n(0)\}_{n=1}^{\infty}$.
\end{corollary}
\begin{proof}
  From the recurrence relation
  \begin{align}
  \label{eqn: three term recurrence relation for GCCS polynomials}
    \hat{\mathcal{P}}_{m+1}(x)=[x-(\gamma_{2m+1}+\gamma_{2m+2})]\hat{\mathcal{P}}_{m}(x)-
    \gamma_{2m-1}\gamma_{2m+2}\hat{\mathcal{P}}_{m-1}(x),
    \quad m\geq0,
  \end{align}
  with $\hat{P}_{-1}(x)=0$ and $\hat{P}_{0}(x)=1$, the chain sequence is given by
  \begin{align*}
   \left\{\dfrac{\gamma_{2n-1}\gamma_{2n+2}}
           {(\gamma_{2n-1}+\gamma_{2n})
           (\gamma_{2n+1}+\gamma_{2n+2})}\right\}_{n=1}^{\infty}
  \end{align*}
  with the parameter sequence
  $\{k'_n\}_{n=0}^{\infty}$=
  $\{\gamma_{2n+2}/(\gamma_{2n+1}\gamma_{2n+2})\}_{n=0}^{\infty}$.
  The result now follows since
  $k'_n=1-g_n$, $n\geq0$.
\end{proof}
The OPS $\{\tilde{\mathcal{P}}_n(x)\}_{n=1}^{\infty}$ can be seen to be
co-recursive with respect to the OPS $\{\hat{\mathcal{P}}_n(x)\}_{n=1}^{\infty}$
arising from the initial conditions
$\tilde{\mathcal{P}}_0(x)=1$ and
$\tilde{\mathcal{P}}_1(x)=\hat{\mathcal{P}}_1(x)+\gamma_2$.
The co-recursive polynomials have been investigated
in the past; see for example,
\cite{Chihara-co-recursive}, and later
\cite{paco-pert-rec-rel-JCAM-1990}
in which the structure and spectrum of the
generalised co-recursive polynomials have been
studied.

Further, from \eqref{eqn: recurrence relation for polynomials in ccs},
the associated chain sequence is $\{\tilde{a}_n^{2}/b_n^{(1)}b_{n+1}^{(1)}\}_{n=1}^{\infty}$
with the first few terms as
\begin{align*}
\dfrac{\tilde{a}_1^2}{b_1^{(1)}b_2^{(1)}}&=
\dfrac{\gamma_4}{(\gamma_3+\gamma_4)}=(1-k_0)k_1;\quad
\dfrac{\tilde{a}_2^2}{b_2^{(1)}b_3^{(1)}}=
\dfrac{\gamma_3\gamma_6}{(\gamma_3+\gamma_4)(\gamma_5+\gamma_6)}=(1-k_1)k_2\\
\dfrac{\tilde{a}_3^2}{b_3^{(1)}b_4^{(1)}}&=
\dfrac{\gamma_5\gamma_8}{(\gamma_5+\gamma_6)(\gamma_7+\gamma_8)}=(1-k_2)k_3
\end{align*}
Proceeding as above, we obtain the minimal parameter sequence $\{k_n\}_{n=0}^{\infty}$ where
$k_0=0$ and
$k_n=\gamma_{2n+2}/b_{n+1}^{(1)}=1-g_n$, $n\geq1$.
which shows that the OPS $\{\tilde{P}_n(x)\}_{n=0}^{\infty}$
is associated with the complementary chain sequence $\tilde{\omega}_n(0)\}_{n=1}^{\infty}$.

Viewing the generalised complementary chain sequences
as perturbations of the
minimal parameters or simply a transformation of the
original chain sequence, we give an important consequence of
Theorem $\ref{thm:main theorem Pn Kn for CCS case}$.
\begin{corollary}
\label{cor: invariance of kernel polynomials in ccs}
The kernel polynomial system $\{\mathcal{K}_n(x)\}$
remains invariant under
generalised complementary chain sequence
if the sequence $\{\gamma_n\}_{n=1}^{\infty}$
satisfies,
\begin{align*}
\gamma_{2n+1}-\gamma_{2n-1}=\gamma_{2n+2}-\gamma_{2n},
\quad n\geq1.
\end{align*}
\end{corollary}
\begin{proof}
The proof follows from a comparison of
\eqref{eqn: recurrence relation for kernel polynomials in ccs}
and the expressions for $b_{n+1}^{(2)}$ and $a_n^{(2)}$.
\end{proof}
Corollary
\ref{cor: invariance of kernel polynomials in ccs}
is important because it is known
\cite[Ex. 7.2, p. 39]{Chihara_book},
that the relation between the monic orthogonal polynomials
and the kernel polynomials is not unique. That is,
for fixed $t\in\mathbb{R}$, though $\{\mathcal{P}_n(x)\}$
will lead to a unique kernel polynomial
system $\{\mathcal{K}_n(t;x)\}$, there are infinite number of
other monic orthogonal polynomial systems which has the same
$\{\mathcal{K}_n(t;x)\}$ as their kernel polynomial system.
Hence generalised complementary chain sequences
can be used to construct
two orthogonal polynomials systems having the same
kernel polynomial systems.

The following theorem unifies the recurrence relations for
the polynomials and the associated kernel polynomials
for both the chain sequence as well as its
generalised complementary
chain sequence.
\begin{theorem}
Consider the recurrence relation,
\begin{align*}
\mathcal{T}_{n}(x)=
(x-\xi_n)\mathcal{T}_{n-1}(x)-
\eta_n\mathcal{T}_{n-2}(x),
\quad n\geq1,
\end{align*}
with $\mathcal{T}_{-1}(x)=0$ and
$\mathcal{T}_{0}(x)=1$.
Then,
\begin{enumerate}[(i)]
\item $\mathcal{T}_{n}(x)\equiv\tilde{\mathcal{P}}_{n}(x)
       (\equiv\hat{\mathcal{P}}_{n}(x))$,
      $n\geq1$, if and only if,
       \begin{align*}
       \xi_1&=\gamma_1(=\gamma_1+\gamma_2),
       \quad
       \xi_{n+1}=b_{n+1}^{(1)},
       \quad n\geq1,
       \quad\mbox{and}\\
       \eta_2&=\gamma_1\gamma_4,
       \quad
       \eta_{n+1}=
       \dfrac{(a_{n-1}^2)^{(2)}(a_n^2)^{(2)}}
             {(a_n^2)^{(1)}},
       \quad n\geq2.
       \end{align*}
\item $\mathcal{T}_{n}(x)\equiv
      \tilde{\mathcal{K}}_{n}(x)$,
      $n\geq1$, if and only if,
       \begin{align*}
       \xi_1&=\gamma_1+\gamma_4,
       \quad
       \xi_{n+1}=b_{n+1}^{(1)}+
       b_{n+2}^{(1)}-b_{n+1}^{(2)},
       \quad n\geq1,
       \quad\mbox{and}\\
       \eta_1&=\gamma_1\gamma_2,
       \quad
       \eta_{n+1}=(a_n^2)^{(2)},
       \quad n\geq1.
       \end{align*}
\end{enumerate}
\end{theorem}

Let the zeros of $\{\mathcal{P}_n(x)\}_{n=0}^{\infty}$
and $\{\tilde{\mathcal{P}}_n(x)\}_{n=0}^{\infty}$
be denoted as
\begin{align*}
0<x_{n,1}<x_{n,2}<\cdots<x_{n,n-1}<x_{n,n}
\quad\mbox{and}\quad 0<\tilde{x}_{n,1}<\tilde{x}_{n,2}<\cdots<\tilde{x}_{n,n-1}<\tilde{x}_{n,n}
\end{align*}
For fixed $n$, by interlacing of zeros of
$\mathcal{P}_n(x)$ and $\tilde{\mathcal{P}}_n(x)$ it is understood that
$x_{n,j}$ are mutually separated by $\tilde{x}_{n,j}$ for $j=1,2,\cdots,n$.
Further, in the present case, it is interesting to note from
\eqref{eqn:general ttrr for polynomial and kernel} and
\eqref{eqn: recurrence relation for polynomials in ccs},
that the sum of the roots of $\mathcal{P}_n(x)$ is given by $\gamma_2+\gamma_3+\cdots+\gamma_{2n}$ while that for $\tilde{\mathcal{P}}_n(x)$ is $\gamma_1+\gamma_3+\cdots+\gamma_{2n}$.
\begin{remark}
It is clear that if $\gamma_1=\gamma_2$, interlacing of the zeros of
$\{\mathcal{P}_n(x)\}$ and $\{\tilde{\mathcal{P}}_n(x)\}$ can never occur.
\end{remark}
For $\gamma_1\neq\gamma_2$, we have the following result.
\begin{proposition}
For fixed $n$, the zeros $\{x_{n,j}\}_{j=1}^{n}$
and $\{\tilde{x}_{n,j}\}_{j=1}^{n}$  cannot interlace if
$(\gamma_1-\gamma_2)$ and $(x_{n,j}-\tilde{x}_{n,j})$ have the same sign
for some $j=1,2,\cdots,n$.
\end{proposition}
\begin{proof}
Suppose $\gamma_1>\gamma_2$ and $x_{n,j}>\tilde{x}_{n,j}$ for some $j=1,2,\cdots,n$.
If the zeros of $\mathcal{P}_n(x)$ and $\tilde{\mathcal{P}}_n(x)$ interlace,
then $\sum_{j=1}^{n}\tilde{x}_{n,j} < \sum_{j=1}^{n}x_{n,j}$
which is a contradiction.

The case $\gamma_1<\gamma_2$ and $x_{n,j}<\tilde{x}_{n,j}$
can be proved similarly.
\end{proof}
The last result in this section shows that while the generalised complementary chain sequence of associated with $\{\hat{\mathcal{P}}_n(x)\}$ yields an OPS, that associated with the associated polynomials $\{\hat{\mathcal{P}^{(1)}}_n(x)\}$ leads to a KOPS.
\begin{theorem}
  Consider the OPS $\{\hat{\mathcal{P}}^{(1)}_n(x)\}_{n=0}^{\infty}$. Then the generalised complementary chain sequence associated with $\hat{\mathcal{P}}^{(1)}_n(x)$ leads to a KOPS $\{\mathcal{Q}_n(x)\}$ satisfying the relation
  \begin{align}
  \label{eqn: three term recurrence relation for kernel with associated polynomials}
  \mathcal{Q}_{n+1}(x)=(x-\gamma_{2n+3}-\gamma_{2n+4})\mathcal{Q}_n(x)-
      \gamma_{2n+2}\gamma_{2n+3}\mathcal{Q}_{n-1}(x),
  \quad n\geq0
  \end{align}
  with $\mathcal{Q}_{-1}(x)=0$ and $\mathcal{Q}_0(x)=1$.
\end{theorem}
\begin{proof}
  It is clear from
  \eqref{eqn: three term recurrence relation for GCCS polynomials}
  that
  $\{\hat{\mathcal{P}}^{(1)}_n(x)\}_{n=1}^{\infty}$ satisfy
  \begin{align*}
  \hat{\mathcal{P}}^{(1)}_{n+1}(x)=(x-\gamma_{2n+3}-\gamma_{2n+4})
   \hat{\mathcal{P}}^{(1)}_{n}(x)-\gamma_{2n+2}\gamma_{2n+3}
    \hat{\mathcal{P}}^{(1)}_{n-1}(x),
    \quad n\geq1
  \end{align*}
  with $\hat{\mathcal{P}}^{(1)}_{-1}(x)=0$ and $\hat{\mathcal{P}}^{(1)}_{0}(x)=1$.
  The associated chain sequence is
  \begin{align*}
  \left\{\dfrac{\gamma_{2n+1}\gamma_{2n+4}}
        {(\gamma_{2n+1}+\gamma_{2n+4})
        (\gamma_{2n+3}+\gamma_{2n+5})}\right\}_{n=1}^{\infty}
  \end{align*}
  with the (non-minimal) parameter sequence
  $\{\gamma_{2n+4}/(\gamma_{2n+3}+\gamma_{2n+4})\}_{n=0}^{\infty}$.
  Hence the OPS $\{\mathcal{Q}_n(x)\}$
  associated with the generalised complementary chain sequence satisfy
  \eqref{eqn: three term recurrence relation for kernel with associated   polynomials}.

  To prove that $\{\mathcal{Q}_{n}(x)\}$ is a KOPS, consider the polynomials $\{\mathcal{U}_{n}(x)\}_{n=1}^{\infty}$ given by
  \begin{align*}
  x\mathcal{Q}_{n}(x)=\mathcal{U}_{n+1}(x)+\gamma_{2n+3}\mathcal{U}_{n}(x),
  \quad n\geq0
  \end{align*}
  The first thing that we require is that
  $\mathcal{U}_{n+1}(0)=-\gamma_{2n+3}\mathcal{U}_{n}(0)$ so that
  choosing $\mathcal{U}_{1}(0)=-\gamma_3$, we have
  $\mathcal{U}_{n+1}(0)=(-1)^{n+1}\gamma_{2n+3}\gamma_{2n+1}\cdots
  \gamma_5\gamma_3$.

  Suppose now that $\{\mathcal{U}_{n}(x)\}_{n=1}^{\infty}$ satisfy the recurrence relation
  \begin{align*}
  \mathcal{U}_{n+1}(x)=(x-\sigma_{n+1})\mathcal{U}_{n}(x)-
   \eta_n\mathcal{U}_{n-1}(x),
  \quad n\geq1,
  \end{align*}
  with $\mathcal{U}_{0}(x)=1$ and $\mathcal{U}_{1}(x)=x-\gamma_3$ and where
  the coefficients $\{\sigma_n\}$ and $\{\eta_n\}$ are to be determined.
  One way for the recurrence relation to hold, is that we must have
  $\mathcal{U}_{n+1}(0)=-\sigma_{n+1}\mathcal{U}_{n}(0)-
  \mu_n\mathcal{U}_{n-1}(0)$
  which implies
  $\gamma_{2n+1}\sigma_{n+1}-\mu_n=\gamma_{2n+3}\gamma_{2n+1}$, $n\geq1$.
  One possible choice for $\sigma_{n+1}$ and $\mu_n$ satisfying the above equations is
  \begin{align*}
  \sigma_{n+1}=\gamma_{2n+3}+\gamma_{2n+2}
  \quad\mbox{and}\quad
  \mu_n=\gamma_{2n+1}\gamma_{2n+2},
  \quad n\geq1.
  \end{align*}
  Since $\mu_n>0$ for $n\geq1$, by Favard's Theorem
  \cite[Theorem 4.4, p. 21]{Chihara_book}
  $\{\mathcal{U}_{n}(x)\}_{n=1}^{\infty}$ becomes a OPS and
  $\{\mathcal{Q}_{n}(x)\}_{n=1}^{\infty}$ its corresponding KOPS
  \cite[eqn. 7.3, p. 35]{Chihara_book}.
\end{proof}
\begin{corollary}
The following holds
\begin{align}
\label{eqn: linear combination of Pn quasi orthogonality}
\lefteqn{x^2\mathcal{Q}_{n}(x)-
\gamma_2[\mathcal{K}^{(1)}_{n+1}(x)+\gamma_{2n+3}\mathcal{K}^{(1)}_{n}(x)]=}
\nonumber\\
&&\mathcal{P}_{n+2}(x)+(\gamma_{2n+3}+\gamma_{2n+4})\mathcal{P}_{n+1}(x)+
\gamma_{2n+2}\gamma_{2n+3}\mathcal{P}_{n}(x),
\quad n\geq1.
\end{align}
\end{corollary}
\begin{proof}
Comparing with the recurrence relation
\eqref{eqn:general ttrr for polynomial and kernel}, it can be observed that
$\{\mathcal{U}_{n}(x)\}_{n=1}^{\infty}$ are co-recursive with respect to
$\{\mathcal{K}_{n}(x)\}_{n=1}^{\infty}$ and arise due to change in the initial value $\mathcal{U}_{1}(x)=\mathcal{K}_{1}(x)+\gamma_2$.
Hence $\mathcal{U}_{n}(x)=
\mathcal{K}_{n}(x)+\gamma_2\mathcal{K}^{(1)}_{n}(x)$, $n\geq1$
\cite{Chihara-co-recursive},
so that
\begin{align*}
x\mathcal{Q}_{n}(x)=
(\mathcal{K}_{n+1}(x)+\gamma_{2n+3}\mathcal{K}_{n}(x))+
\gamma_2(\mathcal{K}^{(1)}_{n+1}(x)+\gamma_{2n+3}\mathcal{K}^{(1)}_{n}(x)).
\end{align*}
Since $x\mathcal{K}_{n}(x)=\mathcal{P}_{n+1}(x)+\gamma_{2n+2}\mathcal{P}_{n-1}(x)$,
$n\geq1$, \eqref{eqn: linear combination of Pn quasi orthogonality} follows.
\end{proof}
\begin{remark}
The polynomials in the left hand side of
\eqref{eqn: linear combination of Pn quasi orthogonality}
are quasi-orthogonal of order 2 on $[0, \infty]$. For details regarding
quasi-orthogonality, the reader is referred to \cite{Brezinski-Driver-ANM-2004-quasi-orthogonality} and the references therein.
\end{remark}
\section{Concluding remarks}
\begin{enumerate}[(A)]

\item
The monic Jacobi matrix of the polynomials
$\mathcal{P}_n(x)$ and $\tilde{\mathcal{P}}_n(x)$
are given respectively by,
\begin{align*}
J_{\mathcal{P}}=
\left(
  \begin{array}{ccccc}
    \gamma_2         & 1                 & 0                 & \cdots \\
    \gamma_2\gamma_3 & \gamma_3+\gamma_4 & 1                 & \cdots \\
    0                & \gamma_4\gamma_5  & \gamma_5+\gamma_6 & \cdots \\
    0                & 0                 & \gamma_6\gamma_7  & \cdots \\
    \vdots           & \vdots            & \vdots            & \ddots \\
  \end{array}
\right)
\quad
J_{\tilde{\mathcal{P}}}=
\left(
  \begin{array}{ccccc}
    \gamma_1         & 1                 & 0                 & \cdots\\
    \gamma_1\gamma_4 & \gamma_3+\gamma_4 & 1                 & \cdots\\
    0                & \gamma_3\gamma_6  & \gamma_5+\gamma_6 & \cdots\\
    0                & 0                 & \gamma_5\gamma_8  & \cdots\\
    \vdots           & \vdots            & \vdots            & \ddots\\
  \end{array}
\right)
\end{align*}
The respective $LU$ decomposition of the above Jacobi matrices
are then given by,
\begin{align*}
L_{\mathcal{P}}=
\left(
  \begin{array}{ccccc}
    1        & 0       & 0       &0      & \cdots \\
    \gamma_3 & 1       & 0       &0      & \cdots \\
    0        &\gamma_5 & 1       &0      & \cdots \\
    0        & 0       &\gamma_7 &1      & \cdots \\
    \vdots   & \vdots  & \vdots  &\vdots &\ddots \\
  \end{array}
\right),
\quad
U_{\mathcal{P}}=
\left(
  \begin{array}{ccccc}
    \gamma_2 & 1       & 0       & 0        & \cdots \\
    0        &\gamma_4 & 1       & 0        & \cdots \\
    0        & 0       &\gamma_6 & 1        & \cdots \\
    0        & 0       & 0       & \gamma_8 & \cdots \\
    \vdots   & \vdots  & \vdots  & \vdots   &\ddots \\
  \end{array}
\right)
\end{align*}
and,
\begin{align*}
L_{\tilde{\mathcal{P}}}=
\left(
  \begin{array}{ccccc}
    1        & 0       & 0       &0      & \cdots \\
    \gamma_4 & 1       & 0       &0      & \cdots \\
    0        &\gamma_6 & 1       &0      & \cdots \\
    0        & 0       &\gamma_8 &1      & \cdots \\
    \vdots   & \vdots  & \vdots  &\vdots &\ddots \\
  \end{array}
\right),
\quad
U_{\tilde{\mathcal{P}}}=
\left(
  \begin{array}{ccccc}
    \gamma_1 & 1       & 0       & 0        & \cdots \\
    0        &\gamma_3 & 1       & 0        & \cdots \\
    0        & 0       &\gamma_5 & 1        & \cdots \\
    0        & 0       & 0       & \gamma_7 & \cdots \\
    \vdots   & \vdots  & \vdots  & \vdots   &\ddots \\
  \end{array}
\right)
\end{align*}
It may be observed that
$L_{\tilde{\mathcal{P}}}$
can be obtained from
$U_{\mathcal{P}}$
by deleting the first column of
$U_{\mathcal{P}}$
while
$U_{\tilde{\mathcal{P}}}$
can be obtained from
$L_{\mathcal{P}}$
by adding the first column of
$U_{\mathcal{P}}$,
with $\gamma_2$ replaced by $\gamma_1$,
as the first column of
$L_{\mathcal{P}}$.
The Jacobi matrix and its LU decomposition
for the polynomial system
$\{\hat{\mathcal{P}}_n(x)\}_{n=1}^{\infty}$
can be obtained similarly.

The monic Jacobi matrix associated with the three canonical transformations,
the Christoffel, the Geronimus and the Uvarov transformations,
are found in
\cite{bueno-paco-darboux-trans-linear-alg-appl-2004},
wherein the procedure is given using the Darboux transformation.
Hence, it will be interesting to study the application of Darboux transformation in the case of (generalised) complementary chain sequences.

\item

The results obtained in this note
can be applied to the Laguerre polynomials
which provide some insights into chain sequences related
to these orthogonal polynomials.

Consider the three term recurrence relation satisfied by
the generalized Laguerre polynomials $\{L_n^{(\alpha)}\}$,
\cite[Page 154]{Chihara_book},
\begin{align}
\label{eqn: recurrence relation for generalized laguerre polynomials}
\mathcal{R}_{n+1}^{(1)}(x)=[x-(2n+\alpha+1)]\mathcal{R}_{n}^{(1)}(x)
-n(n+\alpha)\mathcal{R}_{n-1}^{(1)}(x),
\quad n\geq1,
\end{align}
with
$\mathcal{R}_{0}^{(1)}(x)=1$ and $\mathcal{R}_{1}^{(1)}(x)=x-(1+\alpha)$
and where $\mathcal{R}_{n}^{(1)}(x)\equiv L_n^{(\alpha)}(x)$, $n\geq1$.
Using the notations introduced immediately after
\eqref{eqn:general ttrr for polynomial and kernel},
the associated chain sequence $\{d_n\}_{n=1}^{\infty}$ is,
\begin{align*}
d_n=\dfrac{(a_n^2)^{(1)}}{b_{n}^{(1)}b_{n+1}^{(1)}}=
\dfrac{n(n+\alpha)}{(2n+\alpha-1)(2n+\alpha+1)},
\quad n\geq1,
\end{align*}
and as can be easily verified, the minimal parameters are given by,
$m_n=n/(2n+\alpha+1)$, $n\geq0$.
It is easily seen that
$0<m_n<1/2$, $n\geq1$
and hence by
\cite[Lemma 2.3]{CCS_OPUC},
the chain sequence complementary to
$d_n$ is a SPPCS.
Moreover, for
$-1<\alpha<0$, $m_n/(1-m_n)>n/(n-1)>1$
and hence by Wall's criteria
for SPPCS
\cite[Theorem 19.3]{Wall_book},
$\{d_n\}$ determines its parameters uniquely.
Further, choosing $\gamma_1=0$,
it is found that
$\gamma_2=(1+\alpha)$
and $\gamma_2\gamma_3=1.(1+\alpha)$ implies
$\gamma_3=1$.
Similarly, $\gamma_3+\gamma_4=\alpha+3$ implies
$\gamma_4=\alpha+2$.
Proceeding further on similar lines, it can be easily proved
by induction that
$\gamma_1=0$, $\gamma_{2n}=n+\alpha$ and
$\gamma_{2n+1}=n$, $n\geq1$.
This gives the recurrence relation for the associated kernel
polynomials as
\begin{align}
\label{eqn: recurrence relation for kernel polynomials for Lauguerre}
\mathcal{R}_{n+1}^{(2)}(x)=[x-(2n+\alpha+2)]\mathcal{R}_{n}^{(2)}(x)
-n(n+\alpha+1)\mathcal{R}_{n-1}^{(2)}(x),
\quad n\geq1,
\end{align}
with
$\mathcal{R}_{0}^{(2)}(x)=1$ and $\mathcal{R}_{1}^{(2)}(x)=x-(2+\alpha)$.
Clearly, as is known, $\mathcal{R}_{n}^{(2)}(x)=L_n^{(\alpha+1)}(x)$,
$n\geq1$.

Consider now the polynomials $\{\mathcal{E}_n(x)\}_{n=0}^{\infty}$ satisfying the recurrence relation
\begin{align*}
\mathcal{E}_{n+1}(x)=[x-(2n+\alpha+2)]\mathcal{E}_{n}(x)-
(n+1)(n+\alpha)\mathcal{E}_{n-1}(x),
\quad n\geq1
\end{align*}
with $\mathcal{E}_{0}(x)=1$ and $\mathcal{E}_{1}(x)=x-(\alpha+1)$.
From the related chain sequence, we obtain the sequence $\{\gamma_n\}_{n=1}^{\infty}$
where $\gamma_1=0$, $\gamma_{2n}=n+\alpha$ and $\gamma_{2n+1}=n+1$, $n\geq1$.
The kernel polynomial sequence $\{\mathcal{K}_{n}(x)\}_{n=0}^{\infty}$ associated with $\{\mathcal{E}_{n}(x)\}_{n=0}^{\infty}$ satisfy
\begin{align*}
\mathcal{K}_{n+1}(x)=[x-(2n+\alpha+3)]\mathcal{K}_{n}(x)-
(n+1)(n+\alpha+1)\mathcal{K}_{n}(x),
\quad n\geq0
\end{align*}
with $\mathcal{K}_{-1}(x)=0$ and $\mathcal{K}_{0}(x)=1$.
If we let $\gamma_1=1$, the resulting polynomials satisfy
\begin{align*}
\mathcal{P}_{n+1}(x)=[x-(2n+\alpha+2)]\mathcal{P}_{n}(x)-(n+1)(n+\alpha)
\mathcal{P}_{n-1}(x),
\quad n\geq0
\end{align*}
with $\mathcal{P}_{-1}(x)=0$ and $\mathcal{P}_{0}(x)=1$.
From
\eqref{eqn: recurrence relation for generalized laguerre polynomials}
it is clear that these polynomials are
the associated generalized Laguerre polynomials of order 1 but with $\alpha$ shifted to $\alpha-1$.
%
%The kernel polynomial sequence $\{\mathcal{K}_{n}(x)\}$ for $\{\mathcal{P'}_{n}(x)\}$ satisfy
%
%\begin{align*}
%\mathcal{K}_{n}(x)=[x-(2n+\alpha+1)]\mathcal{K}_{n-1}(x)-
%n(n+\alpha)\mathcal{K}_{n-1}(x),
%\quad n\geq1
%\end{align*}
%
%with $\mathcal{K}_{-1}(x)=0$ and $\mathcal{K}_{0}(x)=1$.
%
The polynomial sequence $\{\hat{\mathcal{P}}_{n}(x)\}$
corresponding to the generalized complementary chain sequence satisfy
\begin{align*}
\hat{\mathcal{P}}_{n+1}(x)=[x-(2n+\alpha+2)]\hat{\mathcal{P}}_{n}(x)-
n(n+\alpha+1)\hat{\mathcal{P}}_{n-1}(x),
\quad n\geq0
\end{align*}
with $\hat{\mathcal{P}}_{-1}(x)=0$ and
$\hat{\mathcal{P}}_{0}(x)=1$.
%$\hat{\mathcal{P}}_{1}(x)=\tilde{\mathcal{P}}_{n}(x)-\gamma_2=
%x-(\alpha+2)$.
Comparing with
\eqref{eqn: recurrence relation for kernel polynomials for Lauguerre},
we find that $\hat{\mathcal{P}}_{n}(x)\equiv L_n^{(\alpha+1)}(x)$, $n\geq1$.

The (co-recursive) polynomials $\{\tilde{\mathcal{P}}_{n}(x)\}$ corresponding to the complementary chain sequence satisfy the recurrence
\begin{align*}
\tilde{\mathcal{P}}_{n+1}(x)=[x-(2n+\alpha+2)]\tilde{\mathcal{P}}_{n}(x)-
n(n+\alpha+1)\tilde{\mathcal{P}}_{n-1}(x),
\quad n\geq1
\end{align*}
with $\tilde{\mathcal{P}}_{0}(x)=1$ and $\tilde{\mathcal{P}}_{1}(x)=x-1$.

Moreover, since the condition in Corollary
\ref{cor: invariance of kernel polynomials in ccs}
is satisfied,
the kernel polynomials for the OPS
$\{\tilde{\mathcal{P}}_{n}(x)\}_{n=0}^{\infty}$
is the same (upto a constant multiple)
as that for the OPS
$\{\mathcal{E}_{n}(x)\}_{n=0}^{\infty}$.

\item

We recall that the Routh-Romanovski Laguerre polynomials $\{N_n^{(p)}(x)\}$  which are analogous to Laguerre polynomials $L_n^{(p)}(x)$
are given by
\begin{align*}
L_n^{(p)}(x)=\dfrac{x^n}{n!}N_n^{(p+2n+1)}\left(\frac{1}{x}\right).
\end{align*}
The sequence $\{N_n^{(p)}(x)\}$ are the polynomials 
studied as one of the three finite classes of 
hypergeometric orthogonal polynomials
\cite{Jamei-three-finite-class-ITSF-2002},
see also 
\cite{Romanovski-1929, Routh-1884}.
%These polynomials $L_n^{(p)}(x)$ are also studied in a different context in RANGA.
%

Further, the sequence $\{N_n^{(p)}(x)\}$ satisfy the three term recurrence
\cite[eqn. 4.19]{Jamei-three-finite-class-ITSF-2002}
\begin{align*}
N_{n+1}^{(p)}(x)=
&\left(\dfrac{(p-(2n+2))(p-(2n+1))}{p-(n+1)}x-
\dfrac{p(p-(2n+1))}{(p-(n+1))(p-2n)}\right)N_{n}^{(p)}(x)\\
&-
\dfrac{n(p-(2n+2))}{(p-(n+1))(p-2n)}N_{n-1}^{(p)}(x)
\end{align*}
The above recurrence relation can be normalised to satisfy a recurrence relation of the form
\eqref{eqn:three term recurrence relation for orthogonal}.
Hence it will be interesting to see the effect of
(generalised) complementary chain sequence on these polynomials.
It is important to remark that the results related to $q$-analogue of Routh-Romanovski polynomials on the unit circle are studied in
\cite{Ranga-arxiv-Routh-Romanovski}.
\end{enumerate}

\end{document}